\documentclass[12pt]{article}
\usepackage[intlimits]{amsmath}
\usepackage{amsfonts,amssymb,amscd,amsthm,cite,hyperref}
\input xy
\xyoption{all}

\usepackage{ifpdf}
\ifpdf 
  \usepackage[pdftex]{graphicx}
  \DeclareGraphicsExtensions{.pdf,.png,.jpg,.jpeg,.mps}
  \usepackage{pgf}
  \usepackage{tikz}
\else 
  \usepackage{graphicx}
  \DeclareGraphicsExtensions{.eps,.bmp}
  \usepackage{pgf}
  \usepackage{tikz}
  \usepackage{pstricks}
\fi

\def\ind{\operatorname{ind}}

\def\rk{\operatorname{rk}}

\newcommand{\im}{\mathop{\rm Im}} 
\newcommand{\coker}{\mathop{\rm coker}}

\def\ch{\operatorname{ch}}
\def\Ell{\operatorname{Ell}}
\def\Mat{\operatorname{Mat}}
\DeclareMathOperator\tr{Tr}

\def\dim{\operatorname{dim}}

\DeclareMathOperator\Mp{Mp}
\DeclareMathOperator\Sp{Sp}
 
\DeclareMathOperator\Ad{Ad}
\DeclareMathOperator\liemp{mp}
\DeclareMathOperator\liesp{sp}

\newtheorem{proposition}{Proposition}
\newtheorem{theorem}{Theorem}

\newtheorem{lemma}{Lemma}

\theoremstyle{definition}
\newtheorem{remark}{Remark}
\newtheorem{definition}{Definition}
\newtheorem{example}{Example}

\def\im{\operatorname{Im}}

\title{An Index Formula for Groups of Isometric Linear Canonical Transformations}

\author{Anton Savin, Elmar Schrohe}

\date{}
\begin{document}

\maketitle

\begin{abstract}
We define a representation of the unitary group $U(n)$ by metaplectic operators acting on $L^2(\mathbb{R}^n)$ and
consider the operator algebra generated by the operators of the representation and pseudodifferential operators of Shubin class. Under suitable conditions, we prove the Fredholm property for elements in this algebra and obtain an index formula.
\end{abstract}

\tableofcontents

\section{Introduction}

Given a representation of a group $G$ on  a space of functions on a manifold $M$, we consider the class of operators equal to   linear combinations of the form
\begin{equation}
\label{gop1}
D=\sum_{g\in G}D_g\Phi_g,
\end{equation}
where the $\Phi_g$ are the operators of the representation, the $D_g$ are pseudodifferential operators on $M$, and we assume that the sum is finite, i.e., only a finite number of $D_g$ is nonzero. 

Operators with shifts (or functional differential operators) are the most widely known examples of operators of the form~\eqref{gop1}. Indeed, suppose that $G$ acts on $M$ by diffeomorphisms $x\mapsto g(x),$ $x\in M,g\in G$. Then we define a representation of $G$ by shift operators $\Phi_g u(x)=u(g^{-1}(x))$. The set of all operators of the form \eqref{gop1} is closed under taking sums and compositions. 
The theory  of $C^*$-algebras was applied to define the notion of ellipticity and to prove the Fredholm property for such operators, see e.g. \cite{AnLe2}; also index formulas were obtained \cite{Ant3,NaSaSt17,SaSchSt2,SaSt30,Per5}. Let us mention  that operators with shifts arise in noncommutative  geometry \cite{Con4,CoDu1,CoMo2,PeRo1,SaSt23},  mechanics \cite{OnTs1,OnSk2,Sku2}, etc.

Recently,  operators of type \eqref{gop1} associated with representations by quantized canonical transformations
on closed manifolds were considered \cite{SaSchSt4,SaSch1}. 
A Fredholm criterion was obtained and an approach to the computation of the index  based on algebraic index theory  was proposed. In a similar vein, an algebraic index theorem was established \cite{GKN}. 
Note that operators associated with quantized canonical transformations arise for example when reducing hyperbolic problems to the boundary~\cite{BaSt1,BolSa1}. 
 
So far, the efforts were limited to the case of compact manifolds.  
In this article, we study operators of type \eqref{gop1} on $\mathbb R^n$ for a particularly 
interesting class of quantized canonical transformations, namely metaplectic operators. 
More precisely, we define a unitary representation of the unitary group $U(n)$ on 
$L^2(\mathbb R^n)$ by metaplectic operators. For a subgroup $G$ of $U(n)$, we consider operators of the form \eqref{gop1}, where the $\Phi_g$ are the metaplectic operators in the representation and the $D_g$ are pseudodifferential operators on $\mathbb{R}^n$ of Shubin type,  see \cite{Shu1} or Section \ref{Sect3}, below, for details.

There are many equivalent definitions of the metaplectic group, see~e.g.~\cite{Ler8,Fol1,deG1}. For instance, it is the group generated by the following three types of operators on $L^2(\mathbb{R}^n)$:
\begin{enumerate}\renewcommand{\labelenumi}{(\roman{enumi})}
\item[$(i)$] $f(x)\longmapsto f(Ax)\sqrt{\det A}$, where $A$ is a real nonsingular $n\times n$ matrix;
\item[$(ii)$] $f(x)\longmapsto f(x)e^{i(Bx,x)}$, where $B$ is a real symmetric $n\times n$ matrix;
\item[$(iii)$] $f(x)\longmapsto \mathcal F(f)(x)$, where $\mathcal{F}$ is the Fourier transform.
\end{enumerate}
Elements of the metaplectic group arise in quantum mechanics as solution operators of nonstationary Schr\"odinger equations with quadratic Hamiltonians \cite{deG1}, also fractional Fourier transforms \cite{BuMa1} are elements of the metaplectic group.

Somewhat surprisingly, the theory becomes rather transparent for this situation. There is a natural notion of ellipticity that implies the Fredholm property.  
Moreover -- and this is the main result in this article -- we obtain 
an index formula valid for all groups $G\subset U(n)$ of polynomial growth in the sense of Gromov~\cite{Gro1}.

This index formula represents the Fredholm index  as a sum of contributions over conjugacy classes in $G$, cf.~\cite{NaSaSt17}. Each contribution is defined in the framework of noncommutative geometry using a certain closed twisted trace (cf.~\cite{CoMo2,Fds15}). The proof of the index formula itself is based on two facts. 
First, the standard index one Euler operator on $\mathbb{R}^n$, defined in terms of the creation and annihilation operators, see~\cite{HKT1}, is actually equivariant with respect to the action of $U(n)$ by metaplectic transformations.  
Second, this operator can be used to derive an equivariant Bott periodicity in the following form 
\begin{equation}\label{eq-bp1}
K_*(C_0(\mathbb{C}^n)\rtimes G)\simeq K_*(C^*G),
\end{equation}
where now $G\subset U(n)$ is an arbitrary subgroup, $C^*G$ is the maximal group $C^*$-algebra of $G$,  
$C_0(\mathbb{C}^n)\rtimes G$ is the maximal $C^*$-crossed product associated with the natural action of $G\subset U(n)$ on $\mathbb{C}^n$ and $K_*$ stands for the $K$-theory of  $C^*$-algebras. Note that the isomorphism~\eqref{eq-bp1}
first appeared in~\cite{HKT1} in terms of $\mathbb{Z}/2$-graded $C^*$-algebras. Here we define this isomorphism in terms of symbols of elliptic operators and give an independent proof of the periodicity isomorphism.
The isomorphism~\eqref{eq-bp1} enables us to reduce the proof of the index formula to the special case of the Euler operator twisted by projections over $C^*G$, where a direct computation of both sides of the index formula is possible.  

{\bf Acknowledgment}.  
We thank Gennadi Kasparov for pointing out the Bott periodicity theorem in~\cite{HKT1} to us. The work of the first author was partly supported  by RUDN University program 5-100; that of the second by DFG through project SCHR 319/8-1.

\section{Isometric Linear Canonical Transformations and Their Quantization}

Let us recall the necessary facts about the  symplectic  and metaplectic groups from \cite{Ler8}, see also \cite{RoRa1,RoSa1}.

\paragraph{The symplectic and the metaplectic groups and their Lie algebras.}

The metaplectic group $\Mp(n)\subset \mathcal{B}L^2(\mathbb{R}^n)$ is the group generated by unitary operators of the form 
$$
\exp(-i\widehat H)\in \Mp(n),
$$
where $\widehat H$ is the Weyl quantization of a homogeneous real quadratic Hamiltonian
$H(x,p)$, $(x,p)\in T^*\mathbb{R}^n$. In its turn, the complex metaplectic group $\Mp^c(n)$ is similarly
generated by unitaries associated with Hamiltonians $H(x,p)+c$, where $H(x,p)$ is as above, while $c$ is a real constant.

The symplectic group $\Sp(n)\subset GL(2n,\mathbb{R}) $ is the group of linear canonical transformations\footnote{i.e., linear transformations that preserve the symplectic form $dx\wedge dp$.} of $T^*\mathbb{R}^n\simeq \mathbb{R}^{2n}$. 
We consider the faithful  representation of this group on $L^2(\mathbb{R}^{2n})$ by shift operators $u(x,p)\mapsto u(A^{-1}(x,p)),$ where $u\in L^2(\mathbb{R}^{2n})$ and $A\in \Sp(n)$ and identify $\Sp(n)$ with its image in $\mathcal{B}L^2(\mathbb{R}^{2n})$ under this representation. One can show that this group is  
generated by the unitary shift operators 
$$
\exp\left(-\left(H_p \frac\partial{\partial x} -H_x \frac\partial{\partial p}\right)\right)\in \Sp(n) 
$$
associated with  the canonical transformation equal to the evolution operator for time $t=1$ of the Hamiltonian system 
$$
 \dot x=H_p,\quad \dot p=-H_x,
$$ 
where $H(x,p)$ is a homogeneous real quadratic Hamiltonian as above.

It is well known that $\Mp(n)$ is a nontrivial double covering of $\Sp(n)$. The projection takes a metaplectic operator to the corresponding canonical transformation. Hence, their Lie algebras, denoted by $\liemp(n)$ and $\liesp(n)$ are isomorphic. Let us describe an explicit isomorphism. Indeed, it follows from the definitions above that
$$
\liemp(n)=\{-i\widehat{H}\},\qquad \liesp(n)=\{-(H_p \partial/\partial x -H_x \partial/\partial p)\}
$$
with $H$ as above and Lie brackets equal to the operator commutators.  These Lie algebras  are  isomorphic   and they are also isomorphic to the Lie algebra of homogeneous real quadratic Hamiltonians $H(x,p)\in \mathbb{R}^{2n^2+n}$ 
\begin{equation}\label{eq-lie2}
\begin{array}{ccccc}
    \liemp(n) & \simeq & \liesp(n) & \simeq & \mathbb{R}^{2n^2+n}\vspace{1mm}\\
 -i\widehat H  & \leftrightarrow & -\left(H_p \frac\partial{\partial x} -H_x \frac\partial{\partial p}\right) & \leftrightarrow & H(x,p),
\end{array}
\end{equation}
where we consider the Poisson bracket on the space of Hamiltonians 
$$
\{H',H''\}=H'_xH''_p-H'_pH''_x.
$$
The fact that the isomorphisms in \eqref{eq-lie2} preserve the Lie algebra structures is proved by a direct computation.

\paragraph{Isometric linear canonical transformations and their quantization.} 
In what follows, we consider the maximal compact subgroup $\Sp(n)\cap O(2n)$ of isometric linear canonical transformations in $\Sp(n)$. 
It is well known that this intersection is isomorphic to the unitary group $U(n)$ if we introduce the complex structure on $T^*\mathbb{R}^n\simeq \mathbb{C}^n$ via $(x,p)\mapsto z= p+ix$, see~\cite{Arn1}. 

If we realize $\liesp(n)$ in terms of Hamiltonians $H(x,p)$ as in~\eqref{eq-lie2}, then one can show that the Lie algebra of the subgroup $\Sp(n)\cap O(2n)\subset \Sp(n)$ consists of the Hamiltonians 
\begin{equation}
\label{eq-hama1}
 H(x,p)=\frac 1 2 (x,p)\left(
                  \begin{array}{cc}
                    A & -B\\
                    B & A\end{array}\right)\left(  \begin{array}{c}   x\\  p\end{array}\right),
\end{equation}
where $A$ and $B$ are real $n\times n$ matrices with $A$  symmetric and $B$ skew-symmetric. Moreover, we have the isomorphism of Lie algebras 
\begin{equation}\label{liealg2}
\begin{array}{ccc}
   \text{Lie algebra of }\Sp(n)\cap O(2n)\subset \Sp(n)& \longrightarrow &u(n) \vspace{2mm}\\
   H(x,p)={\displaystyle \frac 1 2}(x,p)\left(
                  \begin{array}{cc}
                    A & -B\\
                    B & A\end{array}\right)\left(  \begin{array}{c}   x\\  p\end{array}\right)  & \longmapsto & B+iA.
\end{array}
\end{equation}
Here $u(n)$ stands for the Lie algebra of $U(n)$;   its elements are the skew-Hermitian matrices.
Let us illustrate the isomorphism~\eqref{liealg2} in examples.
\begin{example}
If $n=1$, then $B=0$ and $A= \varphi$, $\varphi\in \mathbb{R} $,  and \eqref{eq-hama1} gives  Hamiltonians
$$
H(x,p)=\frac 1 2(x^2+p^2)\varphi.
$$
The solution of the corresponding Hamiltonian system of equations 
$$
\dot x=\varphi p,\; \dot p=-\varphi x; \quad x(0)=x_0,\; p(0)=p_0
$$
satisfies
$$
p(t)+ix(t)=e^{i\varphi t}(p_0+ix_0).
$$ 
For $t=1$ we obtain the element  
$e^{i\varphi}\in U(1)$, obviously equal to the exponential mapping of $i\varphi\in u(1)$. On the other hand, \eqref{liealg2} gives the same element  $B+iA=i\varphi\in u(1)$. 
\end{example}
\begin{example}
If $n=2$, then $A=\left(
                  \begin{array}{cc}
                    k & m\\
                    m & l\end{array}\right)$, $B=\left(
                  \begin{array}{cc}
                    0 & -t\\
                    t & 0\end{array}\right)$ 
and there are four linearly independent Hamiltonians:
$$
x_1^2+p_1^2,\quad  x_2^2+p_2^2,\quad   x_1x_2+p_1p_2,\quad x_1p_2-x_2p_1.
$$
Let us consider  for instance the Hamiltonian
$$
H(x,p) =( x_1p_2-x_2p_1)\varphi=\frac 1 2 (x,p)\left(
                  \begin{array}{cc}
                    A & -B\\
                    B & A\end{array}\right)\left(  \begin{array}{c}   x\\  p\end{array}\right)\varphi,\qquad \varphi\in\mathbb{R},
$$
where $A=0$, $B=\left( \begin{array}{cc}  0& -1\\    1 & 0\end{array}\right)$.
On the one hand, the Hamiltonian system of equations 
$$  
\dot x_1=-\varphi x_2,\; \dot x_2=\varphi x_1,\; \dot p_1=-\varphi p_2,\; \dot p_2=\varphi p_1; \quad x(0) = x_0, p(0) = p_0,
$$
has the solution equal to 
$$
x(t)=e^{B\varphi t}x_0,
 \quad p(t)=e^{B\varphi t}p_0,\quad\text{where }e^{B\varphi t} =\begin{pmatrix}
   \cos (\varphi t) & -\sin(\varphi t)\\
   \sin(\varphi t) & \cos (\varphi t)
 \end{pmatrix}
$$ 
For $t=1$ we therefore obtain
$$
p+ix=e^{B\varphi}(p_0+ix_0).  
$$ 
Then the element $e^{B\varphi}\in U(2)$ is obviously equal to the exponential mapping of $B\varphi\in u(2)$. On the other hand, \eqref{liealg2} gives the same element   $(B+iA)\varphi=B\varphi \in o(2)\subset u(2)$.  
\end{example}

The following lemma will be useful below.
\begin{lemma}\label{lem6}  $U(n)$ is generated by the orthogonal subgroup $O(n)$ and the subgroup $U(1)=\{{\rm diag}(z,1,\ldots,1)\;|\; |z|=1\}$ .
\end{lemma}
\begin{proof}
It suffices to prove that the Lie algebra of $U(n)$ is generated as a vector space by the Lie algebra  of $O(n)$  
and the action of the adjoint representation $\Ad_{O(n)}$ on the Lie algebra of $U(1)$.

Indeed, $u(n)$  
is the set of all matrices $B+iA$, where $A$ is symmetric and $B$ is skew-symmetric. Since $o(n)$  
consists of all skew-symmetric matrices, it suffices to show that the set of all $iA$ is generated by $\Ad_{O(n)}$ of the Lie algebra of $U(1)$. This is straightforward: We first generate diagonal matrices using permutation matrices and then generate nondiagonal matrices using rotations by $\pi/4$ in two-dimensional planes.
\end{proof}

\paragraph{The homomorphism  $ R: U(n) \to \Mp^c(n)$.}

It is known that $\pi:\Mp(n)\to\Sp(n)$ is a nontrivial double covering. Thus, one can not represent unambigously elements of $\Sp(n)$
by metaplectic operators. However, it turns out that one can define  a representation of the unitary subgroup $U(n)\subset \Sp(n)$ by operators in the complex metaplectic group.

\begin{proposition}
Consider the mapping 
\begin{equation}\label{secc1}
  \begin{array}{ccc}
    R: U(n) & \longrightarrow & \Mp^c(n)\vspace{2mm}\\
     s & \longmapsto & \pi^{-1}(s)\sqrt{\det s},
  \end{array}
\end{equation}
defined in a neighborhood of the unit element $I$  in $U(n)$, where $\pi^{-1}$ is the section for $\pi:\Mp(n)\to\Sp(n)$ such that $\pi^{-1}(I)=I$ and the branch of the square root is chosen such that $\sqrt{1}=1$.  Then the mapping \eqref{secc1} extends to the entire group $U(n)$ as a monomorphism of groups. In terms of Hamiltonians, the homomorphism \eqref{secc1}  is defined explicitly as follows. Given
$$
 H(x,p)=\frac 1 2(x,p)\left(
                  \begin{array}{cc}
                    A & -B\\
                    B & A\end{array}\right)
\left(  \begin{array}{c}   x\\  p\end{array}\right),
$$
where $A$ is symmetric and $B$ is skew-symmetric, we have
\begin{equation}\label{eq-quant1}
R\left(\exp( B+iA)\right)=
\exp(-i\widehat{H})\sqrt{\det(\exp(B+iA))}=\exp(-i\widehat{H})\exp(i\tr A/2),
\end{equation}
where $\widehat{H}$ is the Weyl quantization of $H(x,p)$.
\end{proposition}
\begin{proof}
Clearly, this mapping is well defined in a neighborhood of the identity and admits a unique continuation along any continuous path $s(t)$ in $U(n)$, $s(0)=I$ (since this is true for both $\pi^{-1}(s)$ and $\sqrt{\det s}$). Moreover, the result is the same for two homotopic paths with endpoints  fixed. Therefore, to prove that the mapping \eqref{secc1}
is well defined globally, it suffices to check that the continuation along the generators of $\pi_1(U(n))$
gives the same result as the continuation along the constant path. 

It is well known that $ \pi_1(U(n))\simeq \pi_1(U(1))=\mathbb{Z}$ and  a generator is given by the path $s(t)$ equal to rotations in the  $(x_1,p_1)$-plane by angles $t\in[0,2\pi]$. Then we have  
$$
\pi^{-1}(s(t))=e^{-it\widehat{H}},\qquad \widehat{H}=\frac 1 2\left(-\frac{\partial^2}{\partial x_1^2}+x_1^2\right).
$$
Since the spectrum of the harmonic oscillator $\widehat{H}$  is  $\{1/2+k\mid k\in \mathbb N_0\}$,  we see that
$$
\pi^{-1}(s(0))=I,\qquad \pi^{-1}(s(2\pi))=-I.
$$
On the other hand,   $s(t)$ is the diagonal matrix with entries $e^{it},1,...,1$. Hence,
we have
$$
(\det (s (t)))^{1/2}=e^{it/2},
$$
and we see that
$$
R_{s(2\pi)}=\pi^{-1}(s(2\pi))(\det (s (2\pi)))^{1/2}=-I \cdot (-1)=I =R_{s(0)}.
$$
This implies the desired continuity and also smoothness.

Finally, \eqref{eq-quant1} follows from \eqref{secc1} and the fact that the section $\pi^{-1}$ is equal to
$$
\pi^{-1}(\exp( B+iA))=\exp(-i\widehat{H}),
$$
where $\widehat{H}$ is the Weyl quantization of \eqref{liealg2}.
\end{proof}

\section{Elliptic Operators}\label{Sect3}
 
\paragraph{Shubin type pseudodifferential operators.}
We call a smooth function $d=d(x,p)$ on $T^*\mathbb R^n$ a
pseudodifferential symbol (of Shubin type) of order $m\in \mathbb R$, provided its derivatives satisfy the estimates 
$$|D^\alpha_p D^\beta_x d(x,p) |\le c_{\alpha,\beta}(1+|x|+|p|)^{m- |\alpha|- |\beta|}$$
for all multi-indices $\alpha$, $\beta$, with suitable constants $c_{\alpha, \beta}$. 
We moreover assume $d$ to be classical, i.e. $d$ admits an asymptotic expansion $d\sim \sum_{j=0}^\infty
d_{m-j}$, where each $d_{m-j}$ is a symbol of order $m-j$, which is (positively) homogeneous in $(x,p)$ for $|x,p|\ge 1$.

We denote by $\Psi(\mathbb{R}^n)$ the norm closure of the algebra of pseudodifferential operators with Shubin type symbols of order zero   acting on $L^2(\mathbb{R}^n)$, see also \cite{Shu1}. This closure is a $C^*$-subalgebra in $\mathcal{B}(L^2(\mathbb{R}^n))$. The symbol mapping in this situation is the homomorphism
$$
  \begin{array}{ccc}
  \sigma: \Psi(\mathbb{R}^n)& \longrightarrow &  C(\mathbb{S}^{2n-1})\vspace{1mm}\\
D & \longmapsto & \sigma(D)(x,p)
  \end{array}
$$
of $C^*$-algebras, induced by the map which associates to a zero order pseudodifferential operator $D$ with symbol $d\sim \sum_{j=0}^\infty d_{-j}$ the 
restriction of $d_0$ to $\mathbb S^{n-1}$. 
Denoting by $\mathcal K(L^2(\mathbb R^n)) $ the compact operators in $\mathcal B(L^2(\mathbb R^n)) $ we have a short exact sequence 
\begin{eqnarray}\label{ses} 
 0\longrightarrow \mathcal K(L^2(\mathbb R^n)) \longrightarrow \Psi(\mathbb R^n) \stackrel \sigma \longrightarrow C(\mathbb S^{n-1})\longrightarrow 0.
\end{eqnarray}
  
The unitary group $U(n)$ acts on $\Psi(\mathbb{R}^n)$ by conjugation with  metaplectic transformations:
$$
D\in \Psi(\mathbb{R}^n), g\in U(n) \longmapsto R_g D {R}_g^{-1}\in \Psi(\mathbb{R}^n).
$$ 
Moreover, we have an analogue of Egorov's theorem:
$$
 \sigma({R}_g D {R}_g^{-1})={g^{-1}}^*\sigma(D). 
$$

Given a discrete group $G\subset U(n)$, we consider the maximal crossed product  $ \Psi(\mathbb{R}^n)\rtimes  G$  
(for the theory of crossed products, see e.g.~\cite{Ped1,Wil1}).  In the sequel, elements of the crossed product are treated as collections $\{D_g\}_{g\in G}$ of pseudodifferential operators $D_g$. We have a natural representation
\begin{equation}\label{quant3}
 \begin{array}{ccc}
    \Psi(\mathbb{R}^n)\rtimes  G & \longrightarrow & \mathcal{B}(L^2(\mathbb{R}^n)) \vspace{1mm}\\
     \{D_g\} & \longmapsto & \sum_{g\in G} D_g {R}_g.
 \end{array}
\end{equation}
This representation is well defined by the universal property of  the maximal $C^*$-crossed products
and the fact that all operators $R_g$ are unitary.

\paragraph{Operators acting between ranges of projections.}
We next introduce a class of operators that is an analogue of operators acting in sections of vector bundles, cf.~ \cite[Sec.~2.2]{NaSaSt17}.  
Namely, we consider triples  $(D,P_1,P_2)$, where $P_1,P_2$ are $N\times N$ matrix projections over the maximal  group $C^*$-algebra denoted by $C^*(G)$ and $D$ is an $N\times N$ matrix operator over  $\Psi(\mathbb{R}^n)\rtimes G$.  Let us also suppose that $D$ and $P_1,P_2$ are compatible in the sense of the following equality:
$$
  D=P_2DP_1.
$$
If this equality is not satisfied, then we replace  $D$ by $P_2DP_1$.
To any such triple, we assign the operator
\begin{equation}\label{mu-op1}
D:\im P_1\longrightarrow \im P_2,\qquad  \im P_1,\im P_2\subset L^2(\mathbb{R}^n,\mathbb{C}^N),
\end{equation}
called {\em $G$-operator}, where $D,P_1,P_2$ are represented as operators on $L^2(\mathbb{R}^n,\mathbb{C}^N)$ using formula \eqref{quant3}, while $\im P_1,\im P_2$ are the ranges of the projections.  

\paragraph{Ellipticity and Fredholm property.} 
Let us recall the notion of ellipticity in this situation (see~\cite[Sec.~2.2]{NaSaSt17}). 
The symbol homomorphism $\sigma: \Psi(\mathbb{R}^n)\to C(\mathbb{S}^{2n-1})$ induces the symbol homomorphism of the maximal crossed products:
$$
\begin{array}{ccc}
\sigma:  \Psi(\mathbb{R}^n)\rtimes G  & \longrightarrow &  C(\mathbb{S}^{2n-1})\rtimes G\\
\{D_g\} & \longmapsto & \{\sigma(D_g)\}. 
\end{array}
$$

\begin{definition}\label{def-ell1}
A triple $\mathcal{D}=(D,P_1,P_2) $  is {\em elliptic} if there exists an element $r\in\Mat_N(C(\mathbb{S}^{2n-1})\rtimes G)$ such that the following equalities hold
\begin{equation}\label{eq-ell3}
P_1 r  \sigma(D) =P_1,  \quad \sigma(D)    r  P_2 =P_2.
\end{equation}
\end{definition}

\begin{lemma}
Elliptic elements have the Fredholm property. 
\end{lemma}

\begin{proof} The crossed product is an exact functor by \cite[Proposition 3.19]{Wil1}.
Hence the exactness of the short exact sequence \eqref{ses} implies 
the exactness of the corresponding sequence of crossed products by  $G$. In particular, the symbol map $\sigma:\Psi(\mathbb{R}^n)\rtimes G\to C(\mathbb{S}^{2n-1})\rtimes G$ is surjective. 

Given $r$ as in Definition \ref{def-ell1}, we therefore find  $R\in \Mat_N( \Psi(\mathbb{R}^n)\rtimes G) $ with symbol equal to $r$. 
Then \eqref{eq-ell3} implies that 
$$
 P_1R:\im P_2\longrightarrow \im P_1
$$
is a two-sided inverse for \eqref{mu-op1} modulo compact operators.
\end{proof}

\begin{remark}
If $G$ is amenable, then the ellipticity condition can be written more explicitly in terms of the so called trajectory symbol by the results of Antonevich and Lebedev \cite{AnLe2}. Their results apply since the action of $G$ on $\mathbb{S}^{2n-1}$ is topologically free.  Moreover, it turns out that ellipticity is a necessary and sufficient condition for the Fredhom property.
\end{remark}

\section{The Index  Theorem}

\paragraph{Difference construction.}

Given a subgroup $G\subset U(n)$ and an elliptic $G$-operator $\mathcal{D}=(D,P_1,P_2)$, we define the  difference construction  for its symbol
\begin{equation}\label{eq-qq1}
[\sigma(\mathcal{D})]\in K_0(C_0(T^*\mathbb{R}^n)\rtimes G)=K_0(C_0( \mathbb{C}^n)\rtimes G)
\end{equation}
following \cite[Sec.~4.2]{NaSaSt17}. 

Let us recall the construction of the element~\eqref{eq-qq1}. We define the matrix projections 
\begin{equation}\label{projector1}
p_1=\frac12
\begin{pmatrix}
  (1-\sin\psi)P_1  &   \sigma^{-1}(D) \cos\psi \\
  \sigma(D) \cos\psi  & (1+\sin\psi)P_2
\end{pmatrix},\quad 
p_0=\begin{pmatrix}
  0  &  0 \\
  0  & P_2
\end{pmatrix}
\end{equation}
over  the $C^*$-crossed product  $C_0( \mathbb{C}^n)\rtimes G$ with adjoint unit, where  $\sigma^{-1}(D)=r$ (see Definition~\ref{def-ell1}), $\psi=\psi(|z|)\in C^\infty(\mathbb{C}^n)$ is a real $G$-invariant function, which for $|z|$ small  is identically $-\pi/2$, for $|z|$ large is $+\pi/2$, and is nondecreasing. We set
$$
[\sigma(\mathcal{D})]=[p_1]-[p_0].
$$

\begin{remark}\label{diff1}
One can see that  \eqref{projector1} defines a projection also in a more general situation (which is an analogue of the Atiyah--Singer difference construction, see~\cite{AtSi1} or  \cite[Sec.~4.2]{NaSaSt17} in the noncommutative setting). Namely,  consider triples
$$
(a,P_1,P_2),\quad \text{where }a,P_1,P_2\in\Mat_N(C(\mathbb{C}^n)\rtimes G),
$$
where $P_1$ and $P_2$ are projections, $a=P_2aP_1$, and the triple is elliptic in the sense of Definition~\ref{def-ell1} for $|x|^2+|p|^2$ large. More precisely, we require that for the restriction of the triple $(a,P_1,P_2)$ to a subset of the form $\{(x,p)\in \mathbb{C}^n\;|\; |x|^2+|p|^2\ge R^2\}$ for some $R>0$ there exists a triple $(r,P_2,P_1)$ such that $ra=P_1$ and $ar=P_2$ (cf.~\eqref{eq-ell3}). Then, if we replace the triple $(\sigma(D),P_1,P_2)$   in \eqref{projector1} by the triple $(a,P_1,P_2)$, then the difference of projections~\eqref{projector1} gives a well-defined class in $K$-theory. Of course, such triples are not  in general  symbols of $G$-operators.
\end{remark}

\paragraph{Homotopy classification.}

Two elliptic $G$-operators $(D_0,P_0,Q_0)$ and $(D_1,P_1,Q_1)$ as in \eqref{mu-op1} are called {\em homotopic} if there exists a continuous homotopy of elliptic operators $(D_t,P_t,Q_t)$, $t\in [0,1]$, which gives the original operators for $t=0$ and $t=1$.
Two elliptic operators are called {\em stably homotopic} if their direct sums with some trivial operators are homotopic. Here trivial operators are operators of the form $(1,P,P)$, where $P$ is a projection. It turns out that stable homotopy is an equivalence relation on the set of elliptic operators. The set of equivalence classes of elliptic $G$-operators is denoted by $\Ell(\mathbb{R}^n,G)$. This set is an Abelian group, where the sum corresponds to the direct sum of operators and the zero of the group is equal to the equivalence class of trivial operators.    

The difference construction  \eqref{eq-qq1} induces the mapping
 \begin{equation}\label{homo1}
 \begin{array}{ccc}
 \Ell(\mathbb{R}^n,G) & \longrightarrow & K_0(C_0(T^*\mathbb{R}^n)\rtimes G),\vspace{1mm}\\
   \mathcal{D}=(D,P_1,P_2)  & \longmapsto & [\sigma(\mathcal{D})].
\end{array}
\end{equation}
\begin{proposition}
The mapping \eqref{homo1} is an isomorphism of Abelian groups.
\end{proposition}
The proof is standard, see \cite[Sec.~4.3]{NaSaSt17} or \cite{Sav8}.

\paragraph{Smooth symbols.}

In this paper, we obtain a cohomological index formula. To this end, we use methods of noncommutative geometry and have to assume that our symbol is smooth in a certain sense. More precisely, we make the following assumption. From now on we suppose that $G\subset U(n)$  is a discrete group of polynomial growth~\cite{Gro1}.  Under this assumption, one can define smooth crossed products by actions of $G$, which are  spectrally invariant in the corresponding $C^*$-crossed products (see \cite{Schwe1}).

Recall that the smooth crossed product ${A}\rtimes G$ of a Fr\'echet algebra ${A}$ with the seminorms $\|\cdot\|_m$, $m\in\mathbb{N}$, and a group $G$ of polynomial growth acting on ${A}$ by automorphisms $a\mapsto g(a)$ for all $a\in {A}$ and $g\in G$ is equal to the vector space of  collections $\{a_g\}_{g\in G}$ of elements in $A$ that decay rapidly at infinity in the sense that the following estimates are valid:
$$
\|a_g\|_m\le C_N(1+|g|)^{-N}\quad \text{for all $N,m\in  \mathbb{N},$ and $g\in G$},
$$
where the constant $C_N$ does not depend on $g$. Here $|g|$ is the length of $g$ in the word metric on  $G$. Finally, the action of $G$ on ${A}$ is required to be tempered: for any $m$ there exists $k$ and a   polynomial $P(z)$ with positive coefficients such that $\|g(a)\|_m\le P(|g|)\|a\|_k$ for all $a$ and $g$. 
The product in ${A}\rtimes G$ is defined by the formula: 
$$
\{a_g\}\cdot \{b_g\}=\left\{\sum_{g_1g_2=g}a_{g_1} g_1(b_{g_2})\right\}.
$$

It follows from the results in \cite{Schwe1} that the group $K_0(C_0(\mathbb{C}^n)\rtimes G)$ is isomorphic to the group of stable homotopy classes of elliptic symbols $(\sigma(D),P_1,P_2)$ that are {\em smooth} in the following sense: their components lie in the  smooth crossed products
\begin{equation}\label{ssymb1}
\sigma(D)\in \Mat_N(C^\infty(\mathbb{S}^{2n-1})\rtimes G), \quad P_1,P_2\in \Mat_N(C^\infty(G)).
\end{equation}
Here the smooth group algebra $C^\infty(G)$ is interpreted as the smooth crossed product $\mathbb{C}\rtimes G$.

Our aim is to define the topological index for smooth elliptic symbols. 

\paragraph{Algebraic preliminaries.}

Suppose that a group $G$ acts by automorphisms on a differential graded algebra $A$ with the differential denoted by $d$.

\begin{definition}[cf. \cite{Fds15,CoMo2}]\label{twtr1}
Given $s\in G$, a {\em closed twisted trace} is a linear functional
$$
\tau_s: A\longrightarrow \mathbb{C}
$$
such that 
\begin{itemize}
\item $\tau_s(ab)=\tau_s(bg(a))(-1)^{\deg a \deg b}$ for all $a,b\in A$.
\item $\tau_s(da)=0$ for all $a \in A$. 
\end{itemize} 
Two twisted traces $\tau_s$ and $\tau_{gsg^{-1}}$ are {\em compatible} if $\tau_{gsg^{-1}}(a)=\tau_s(g^{-1}a)$ for all $a\in A$.
\end{definition}
\begin{example}
Let the elements of $A$ and $G$ be represented by operators $a$ and $U_g$ on some Hilbert space. Then we can set
$$
\tau_s(a)=\tr(U_s a)\quad\text{ for all }a\in A,
$$
provided that the operator trace $\tr$ exists. Then this collection of functionals is a compatible collection of twisted traces.
\end{example}

Given a compatible collection of twisted traces and a conjugacy class $\langle s \rangle\subset G$, we define the functional
$$
\tau_{\langle s \rangle}: A\rtimes G \longrightarrow \mathbb{C} 
$$
on the algebraic crossed product of $A$ and $G$ by the formula
\begin{equation}\label{eq-tra1}
 \tau_{\langle s \rangle}\{a_g\}=\sum_{g\in \langle s \rangle}\tau_g(a_g).
\end{equation}
We claim that this functional is a trace, i.e., we have 
$$
 \tau_{\langle s \rangle}(ab)=\tau_{\langle s \rangle}(ba)(-1)^{\deg a \deg b}\quad
 \text{for all $a,b\in A\rtimes G$.}
$$
Indeed,  if both $a$ and $b$ have a single nonzero component denoted by $a_g$ and $b_h$ respectively,  then $ab$ and $ba$ also have a single nonzero component equal to $a_g g(b_h)$ and $b_h h(a_g)$, and we have
\begin{multline}
\tau_{\langle gh\rangle}(ab)=\tau_{  gh }(a_g g(b_h))=\tau_{  g(hg)g^{-1} }(a_g g(b_h))=\\
=\tau_{   hg }(g^{-1}(a_g)  b_h)=\tau_{   hg }(b_h h(a_g)   )(-1)^{\deg a \deg b}=\tau_{\langle gh\rangle}(ba )(-1)^{\deg a \deg b}.
\end{multline}

\paragraph{Twisted traces on differential forms.}

Let $G=U(n)$ act on $\mathbb{C}^n\simeq T^*\mathbb{R}^n$ and consider the induced action on differential forms $C^\infty_c(\mathbb{C}^n,\Lambda(\mathbb{C}^n))$ considered as a differential graded algebra.  We now construct a compatible collection of closed twisted traces for all elements of the unitary group.  To this end, 
given $s\in U(n)$, we define the orthogonal decomposition
$$
 \mathbb{C}^n=L=L_s\oplus L_s^\perp,
$$
where $L_s$ is the fixed point subspace of $s$ (equivalently, it is the eigensubspace associated with eigenvalue $1$),
while $L_s^\perp$ is its orthogonal complement. Then we define the functional 
$$
\begin{array}{ccc}\label{geom1}
\tau_s: C^\infty_c(\mathbb{C}^n,\Lambda(\mathbb{C}^n)) & \longrightarrow  & \mathbb{C}\vspace{2mm}\\
\omega& \longmapsto & \displaystyle \tau_s(\omega)=\int_{L_s } \omega|_{L_s } .
\end{array}
$$
Here, we use the complex orientation on $L_s$ (if $z_j=p_j-ix_j$ are the complex coordinates, then $\prod_j dp_j\wedge dx_j$ is assumed to be positive). Clearly, this definition does not depend on the choice of coordinates $z$. Moreover, these functionals define a compatible collection of twisted traces in the sense of Definition~\ref{twtr1}.  

Thus, for each $s\in G$ we get (see~\eqref{eq-tra1}) a closed graded trace
$$
 \tau_{\langle s\rangle}:C^\infty_c(\mathbb{C}^n,\Lambda(\mathbb{C}^n))\rtimes U(n) \longrightarrow \mathbb{C}
$$
on the algebraic crossed product.

\paragraph{The definition of the topological index.} Let us define the topological index as the functional
$$
\ind_t: K_0(C_0( \mathbb{C}^n)\rtimes G)\longrightarrow \mathbb{C}.
$$
To this end, we represent classes in the latter $K$-group as formal differences $[P_1]-[P_0]$ of projections in the smooth crossed product $\Mat_N(C^\infty(\mathbb{C}^n)\rtimes G)$ such that $P_1=P_0$ at infinity in $\mathbb{C}^n$. Then we set
\begin{equation}\label{indt2}
\ind_t([P_1]-[P_0])
=\sum_{\langle s\rangle\subset G} \frac{1}{\det(1-s|_{L_s^\perp})}{\rm tr }\;\tau_{\langle s\rangle}
\left(P_1\exp\left(-\frac{ dP_1dP_1}{2\pi i}\right)-P_0\exp\left(-\frac{ dP_0dP_0}{2\pi i}\right) \right).
\end{equation}
(cf.~\cite{AtSe2}). Here the summation is over the set of all conjugacy classes $\langle s\rangle\subset G$ and ${\rm tr}$ stands for the matrix trace. Note that each summand in   \eqref{indt2} is  homotopy invariant. We refer to this invariant as the {\em topological index localized at the conjugacy class} $\langle s\rangle\subset G$ and denote it by $\ind_t([P_1]-[P_0])(s)$.

\paragraph{The index theorem.}

\begin{theorem}
Given an elliptic $G$-operator $\mathcal{D}=(D,P_1,P_2)$ associated with a discrete group $G\subset U(n)$ of polynomial growth, the following index formula holds
\begin{equation}\label{indf1}
 \ind \mathcal{D}=\ind_t[\sigma(\mathcal{D})].
\end{equation}
\end{theorem}

The idea of our proof is to use the homotopy invariance of both sides of the index formula and to use $K$-theory to reduce the operator  to a very special operator,  for which one can compute both sides of the index formula independently and check that they are equal.

\section{The Euler Operator and Equivariant Bott Periodicity} 

The aim of this section is to  define an isomorphism of Abelian groups
$$
 \beta: K_0(C^*(G)) \longrightarrow K_0(C_0(T^*\mathbb{R}^n)\rtimes G).
$$
This isomorphism will be defined in terms of   the Euler operator on $\mathbb{R}^n$.
If $G$ is trivial, then this isomorphism coincides with the classical Bott periodicity isomorphism. For nontrivial groups, this isomorphism is a variant of equivariant Bott periodicity. Note also that if $G\subset O(n)$, then this isomorphism was constructed in \cite{NaSaSt17}.

\paragraph{Euler operator.} Recall that the \emph{classical Euler operator}  on a Riemannian manifold $M$ is
defined by the formula
\begin{equation}\label{eulerrr1}
d+d^*\colon C^\infty(M,\Lambda^{ev}(M))\longrightarrow
C^\infty(M,\Lambda^{odd}(M)).
\end{equation}
It takes differential forms of even degree to differential forms of
odd degree. Here  $d$ is the exterior derivative and $d^*$ is its
adjoint with respect to the Riemannian volume form and the inner
product on forms defined by the Hodge star operator. Let us modify this operator and
obtain the following elliptic operator in $\mathbb{R}^n$ (e.g., see~\cite{HKT1})
\begin{equation}\label{euler2}
\mathcal{E}=d+d^*+xdx\wedge +(xdx\wedge)^*\colon \mathcal{S}(\mathbb{R}^n,\Lambda^{ev}(\mathbb{C}^n))\longrightarrow
\mathcal{S}(\mathbb{R}^n,\Lambda^{odd}(\mathbb{C}^n)).
\end{equation}
Here $xdx=dr^2/2=\sum_j x_jdx_j$, where $r=|x|$. 
Its symbol is invertible for $|x|^2+|p|^2\ne0$.\footnote{Indeed, $\sigma(\mathcal{E})(x,p)=(ip+xdx)\wedge+((ip+xdx)\wedge)^*$.  Hence,
$\sigma(\mathcal{E})^2(x,p)=(|x|^2+|p|^2)Id.$} 
We consider this operator in the Schwartz spaces of complex valued differential forms. 

The following lemma is well known.
\begin{lemma}
The kernel $\ker \mathcal{E}$ can be identified with $\mathbb{C}e^{-|x|^2/2}$, while $\coker\mathcal{E}=0$.
\end{lemma}

\begin{example}
If $n=1$, then 
\begin{equation}\label{euler2a}
\mathcal{E}=\frac{\partial}{\partial x}+x\colon \mathcal{S}(\mathbb{R})\longrightarrow
\mathcal{S}(\mathbb{R})
\end{equation}
is just the annihilation operator modulo $\sqrt{2}$ (here we skip $dx$ in the differential forms in the target space).
\end{example}

It follows from the definition that $\mathcal{E}$ is $O(n)$-equivariant with respect to the natural action of $O(n)$ on differential forms. Let us show that  $\mathcal{E}$  is   equivariant with respect to   $U(n)$.

\paragraph{A unitary representation $\rho:U(n)\to \mathcal{B}(L^2(\mathbb{R}^n,\Lambda(\mathbb{C}^n)))$.}  

The identification $\Lambda(\mathbb{R}^n)\otimes\mathbb{C}\simeq \Lambda(\mathbb{C}^n)$ yields a  unitary representation
$$U(n)\longrightarrow {\rm Aut}(\Lambda(\mathbb{R}^n)\otimes\mathbb{C}),$$
namely the natural representation on the algebraic forms:
$$
g\in U(n),\omega\in \Lambda(\mathbb{C}^n) \longmapsto {g^*}^{-1}\omega,
$$
which is well defined since we consider complex valued forms. 

Complementing the unitary representation ${R}: U(n)\longrightarrow \mathcal{B}(L^2(\mathbb{R}^n))$ introduced in~\eqref{secc1}
we next define the  representation
\begin{equation}\label{eq-diag1}
\rho: U(n)\longrightarrow \mathcal{B}(L^2(\mathbb{R}^n,\Lambda(\mathbb{C}^n)))
\end{equation}
as the diagonal representation:
$$
\rho_g\Bigl(\sum_I \omega_I(x)dx^I\Bigr)= \sum_I {R}_g(\omega_I){g^*}^{-1}(dx^I), \quad g\in U(n),
$$
where we represent differential forms as sums $\sum_I \omega_I(x)dx^I$ over multi-indices 
with $L^2$ coefficients $\omega_I(x)$. 

As a tensor product of unitary representations, $\rho$ is a unitary representation. 

\paragraph{$U(n)$-equivariance of the Euler operator.}
Note that for $g\in O(n)\subset U(n)$ we have $R_g={g^*}^{-1}$, hence in this case $\rho_g={g^*}^{-1}$   is just the natural action of $g$ on differential forms. 
\begin{lemma}
$\mathcal{E}$ is $U(n)$-equivariant, i.e., we have
\begin{equation}\label{eq-equiv1}
\rho_g\mathcal{E}\rho_g^{-1}=\mathcal{E}\text{ for all }g\in U(n).
\end{equation}
\end{lemma}
\begin{proof}
By Lemma~\ref{lem6}, $U(n)$ is generated by $O(n)$ and $U(1)$. Thus, it suffices to prove \eqref{eq-equiv1} for $g$ in one of these two subgroups.
For $g\in O(n)$, this equality follows from the definition (since $d,d^*, dr^2$ commute with the action of the orthogonal group by shifts). Thus, it remains to prove the statement for $g\in U(1)$. For simplicity, we consider the one-dimensional case (the general case is treated similarly).

Let $n=1$. Then we know  (see \eqref{secc1})
$$
{R}_g=e^{it(1/2-\widehat{H})},\qquad \text{where }g=e^{it}\in U(1).
$$
It is easy to see that
$$
{g^*}^{-1}|_{\Lambda^0(\mathbb{R})}=1,\qquad {g^*}^{-1}|_{\Lambda^1(\mathbb{R})}=e^{-it}.
$$
Hence, the desired equivariance amounts to proving that the   operator
$$
 \frac{\partial}{\partial x}+x
$$
has the property
$$
 e^{it(-1/2-\widehat{H})}\left(\frac{\partial}{\partial x}+x\right)e^{-it(1/2-\widehat{H})}=\frac{\partial}{\partial x}+x.
$$
Let us prove this identity by Dirac's method. We  define creation and annihilation operators
$$
A^*=\frac{1}{\sqrt{2}}\left(-\frac{\partial}{\partial x}+x\right),\qquad A=\frac{1}{\sqrt{2}}\left(\frac{\partial}{\partial x}+x\right)
$$
One also has $\widehat{H}=AA^*-1/2=A^*A+1/2$. This enables us to show that
$$
e^{-it\widehat{H}}A=e^{it/2}e^{-itAA*}A=e^{it/2}Ae^{-itA^*A}=e^{it/2}Ae^{-it(\widehat{H}-1/2)}=
e^{it}Ae^{-it\widehat{H}}.
$$
Hence, we get 
$$
e^{it(-1/2-\widehat{H})}A e^{-it(1/2-\widehat{H})}=e^{-it} e^{-it\widehat{H}}Ae^{it\widehat{H}}=e^{-it}e^{it}A=A.
$$
This completes the proof of equivariance for $n=1$.
\end{proof}

\paragraph{Twisting by a projection.}

Let $P=(P_g)\in\Mat_N(C^*(G))$ be a projection over the group $C^*$-algebra of  $G\subset U(n)$. Then we define a projection
$$
1\otimes P: L^2(\mathbb{R}^n,\Lambda(\mathbb{C}^n)\otimes \mathbb{C}^N)\longrightarrow L^2(\mathbb{R}^n,\Lambda(\mathbb{C}^n)\otimes \mathbb{C}^N)
$$
by the formula  
$$
1\otimes P=\sum_{g\in G} (1\otimes P_g) (\rho_g\otimes 1_N).
$$
The map $P\mapsto 1\otimes P$ is defined by a covariant representation, hence, it gives a homomorphism of $C^*$-algebras. This implies that $1\otimes P$ is a projection.

Since $1\otimes P$ is a projection, its range is a closed subspace denoted by $\im (1\otimes P)$. Thus, we can define the twisted operator as
\begin{equation}\label{euler3}
 \mathcal{E}_0\otimes 1_N:\im (1\otimes P)\longrightarrow\im (1\otimes P),
\end{equation}
where we made a reduction to  the zero-order operator 
$$
\mathcal{E}_0=(\mathcal{E}\mathcal{E}^*+1)^{-1/2}\mathcal{E}.
$$
Since $\mathcal{E}$ is equivariant, it follows that $(\mathcal{E}_0\otimes 1_N)(1\otimes P)=(1\otimes P)(\mathcal{E}_0\otimes 1_N)$. Thus,  $\mathcal{E}_0\otimes 1_N$
preserves $\im (1\otimes P)$. This twisted operator is Fredholm with an almost inverse operator equal to $\mathcal{E}^{-1}_0\otimes 1_N$.

\paragraph{Equivariant Bott periodicity.}

\begin{theorem}\label{bott2}
The mapping
\begin{equation}\label{main4}
 \begin{array}{ccc}
 \beta: K_0(C^*(G))   & \longrightarrow &  K_0(C_0(T^*\mathbb{R}^n)\rtimes G)\\
P & \longmapsto &  [(\sigma(\mathcal{E}_0\otimes 1_N), 1\otimes P ,1\otimes P)]
 \end{array}
\end{equation}
is an isomorphism of Abelian groups. 
\end{theorem}

\begin{proof}

The idea (going back to Atiyah \cite{Ati7}) is to include the mapping $\beta$ in the diagram:
\begin{equation}\label{diaga5}
 \xymatrix{
   K_0(C^*(G)) \ar[r]^\beta & K_0(C_0(T^*\mathbb{R}^n)\rtimes G) \ar[r]^{\beta'} \ar@/^2pc/[l]^{\ind} & K_0(C_0(\mathbb{R}^{4n})\rtimes G) \ar@/^2pc/[l]^{\ind'} 
}
\end{equation}
of Abelian groups and homomorphisms with the following properties
\begin{eqnarray}
\ind\circ \beta=I, \label{q1}\\
\ind'\circ\beta'=I ,\label{q2}\\
\ind'\circ \beta'=\beta\circ \ind. \label{q3}
\end{eqnarray}
Clearly, if we construct the diagram with these properties, then $\beta$ and $\ind$ are mutually inverse homomorphisms and the theorem is proved. 

It remains to construct the diagram with these properties.
The mappings $\beta,\beta'$ will be defined by taking exterior products with the Euler operator, while   $\ind,\ind'$
will be analytic index mappings. Hence, properties \eqref{q1} and \eqref{q2} follow from the multiplicative property of the index and the fact that the index of the Euler operator is equal to one. It turns out that the remaining property \eqref{q3} also follows from the multiplicative property of the index and an explicit homotopy of symbols (the so-called Atiyah rotation trick \cite{Ati7}). Let us now give the detailed proof. 

\paragraph{1. Definition of the mapping $\beta'$.} 
Consider the doubled space
\begin{equation*}
    \mathbb{R}^{4n}=\mathbb{R}^{2n}\times\mathbb{R}^{2n},\quad (x,p,y,q)\in \mathbb{R}^{4n},
\end{equation*}
with the diagonal action of $G$ on it. Let us define the  mapping
\begin{align*}
    \beta'\colon K_0(C_0(T^*\mathbb{R}^n)\rtimes G)&\longrightarrow K_0(C_0(\mathbb{R}^{4n})\rtimes G), \\
     {}[\sigma]&\longmapsto[\sigma\#\sigma(\mathcal{E})_0] 
\end{align*}
in terms of the exterior product of symbols, see \cite[Sec.~6.2]{NaSaSt17}. 
We recall the definition of the exterior product.
To this end, let $\mathbf{a}=(a,P_1,P_2)$ and $\mathbf{b}=(b,Q_1,Q_2)$ be triples over
 $C(\mathbb{R}^{2n})\rtimes G$  and $C(\mathbb{R}^{2n})$  respectively,  see Remark~\ref{diff1}.
Suppose in addition that  the projections $P_j,Q_j$ are self-adjoint, and $\mathbf{b}$ is
equivariant. This means that we have two homomorphisms  $\rho_j:G\to {\rm End}(\im Q_j)$ from $G$  to the group of unitary automorphisms of the vector bundles equal to the ranges of $Q_1$ and $Q_2$, and $b$ intertwines these homomorphisms: $b\rho_1(g)=\rho_2(g)b$ for all $g$.
\begin{definition}\label{cross-sy1}
The \textit{exterior product} of the triples  $\mathbf{a}$ and
$\mathbf{b}$ is the triple   
\begin{equation*}
    \mathbf{a}\#\mathbf{b}=\left(\begin{pmatrix}
      a\otimes 1 & -1\otimes b^* \\
      1\otimes b & a^*\otimes 1
    \end{pmatrix}, \begin{pmatrix}P_1\otimes Q_1\!\!\!\!\!\!&0
      \\
      \,\,\,0&P_2\otimes Q_2
    \end{pmatrix}, \begin{pmatrix}P_2\otimes Q_1\!\!\!\!\!\!&0
      \\
      \,\,\,0&P_1\otimes Q_2
    \end{pmatrix}\right)
\end{equation*}
over $C (\mathbb{R}^{4n})\rtimes G$. Here the elements of these matrices are in matrix algebras over the crossed product $C (\mathbb{R}^{4n})\rtimes G$ and they are defined as 
$$
 (a\otimes 1)_g=a_g\otimes \rho_1(g),\quad (a^*\otimes 1)_g=(a^*)_g\otimes \rho_2(g), \quad 
 (P_j\otimes Q_k)_g=P_{jg}\otimes \rho_k(g)Q_k.
$$
We shall frequently abridge this notation and simply write
\begin{equation*}
    a\#b=\begin{pmatrix}
      a&-b^* \\
      b & a^*
    \end{pmatrix},
\end{equation*}
omitting the projections and tensor products by identity operators.
\end{definition}

\begin{lemma}
Suppose that the triples $\mathbf{a}=(a,P_1,P_2)$ and $\mathbf{b}=(b,Q_1,Q_2)$ are elliptic. Then their   exterior product $\mathbf{a}\#\mathbf{b}$   is elliptic.
\end{lemma}
 
\begin{proof}
1. Let us state the ellipticity condition for triples in $C^*$-algebraic terms. Consider a triple $(a,P_1,P_2)$ with components in a $C^*$-algebra $A$, $P_j=P_j^*=P_j^2$, and $a=P_2aP_1$. Such a triple is elliptic if there exists  $r\in A$ such that $ar=P_2$ and $ra=P_1$.  We claim that the ellipticity is equivalent to the  following two conditions
\begin{equation}\label{eq-adj1}
\begin{array}{c}
 aa^* \text{ is invertible in the $C^*$-algebra }  P_2AP_2,\\
 a^*a \text{ is invertible in the $C^*$-algebra }  P_1AP_1.
\end{array}
\end{equation}
The proof is standard. Namely.  ellipticity of $(a,P_1,P_2)$ is equivalent to that of $(a^*,P_2,P_1)$ and, hence, to that of 
$$
\left(
\begin{pmatrix}
      0& a^* \\
      a & 0
    \end{pmatrix},
\begin{pmatrix}
      P_1& 0 \\
      0 & P_2
    \end{pmatrix},
\begin{pmatrix}
      P_1& 0 \\
      0 & P_2
    \end{pmatrix}
\right).
$$
Further, the ellipticity of this matrix triple is equivalent to the invertibility of the matrix $\begin{pmatrix}
      0& a^* \\
      a & 0
    \end{pmatrix}$ in the algebra $P_1AP_1\oplus P_2AP_2$.  Finally the invertibility of this self-adjoint matrix is equivalent  to the invertibility of its square, which gives the desired result. 

2. Thus, to prove the lemma, it suffices to prove the  invertibility of $(a\# b)^*(a\# b)$ and $(a\# b)(a\# b)^*$ in the corresponding $C^*$-algebras. Let us prove that the first element is invertible. The verification for the second element is similar. Using the equivariance of $\mathbf{b}$, we obtain that the off-diagonal elements in $a\# b$ and $(a\# b)^*$ commute with the elements on the diagonal. This implies that the composition  
\begin{equation}\label{eq-prod3}
(a\# b)^*(a\# b)\\
={\rm diag}(a^*a\otimes 1+1\otimes b^*b, aa^*\otimes 1+1\otimes bb^*)
\end{equation}
is a diagonal matrix.  Let us prove that the upper left corner of this matrix is invertible (the invertibility of the lower right corner is proved similarly) in the algebra
\begin{equation}\label{eq-elem2}
a^*a\otimes 1+1\otimes b^*b\in   (P_1\otimes Q_1 )
\Mat_{N^2}(C(\mathbb{R}^{4n})\rtimes G)
  (P_1\otimes Q_1).  
\end{equation} 
Let us denote this element and the algebra in \eqref{eq-elem2} as $u$ and $\mathcal{A}(\mathbb{R}^{4n})$ respectively.
Moreover, given a $U(n)$-invariant closed subset $U\subset\mathbb{R}^{4n}$, we denote the corresponding algebra by $\mathcal{A}(U)$.

Since $(a,P_1,P_2)$ is elliptic on the set $\{|x|^2+|p|^2\ge R^2\}$  and $1\otimes b^*b$ is nonnegative, it follows that  the element \eqref{eq-elem2}
is invertible in the algebra $\mathcal{A}(\mathbb{R}^{4n}\cap\{|x|^2+|p|^2\ge R^2\})$ as a sum of nonnegative elements, one of which is invertible. Denote by $r_1\in \mathcal{A}(\mathbb{R}^{4n}\cap\{|x|^2+|p|^2\ge R^2\})$ the inverse element and by $\widetilde r_1\in \mathcal{A}(\mathbb{R}^{4n})$ a lift under the projection mapping 
$$
   \mathcal{A}(\mathbb{R}^{4n})\longrightarrow \mathcal{A}(\mathbb{R}^{4n}\cap\{|x|^2+|p|^2\ge R^2\}).
$$
Such a lift exists by the exactness of the maximal crossed product functor \cite[Proposition 3.19]{Wil1}. Then the differences
\begin{equation}\label{eq-r1}
u \widetilde r_1-1, \widetilde r_1 u-1
\end{equation}
vanish in the domain $ \{|x|^2+|p|^2\ge R^2\}$.
Similarly, using the ellipticity of $(b,Q_1,Q_2)$, we obtain an element $\widetilde r_2\in \mathcal{A}(\mathbb{R}^{4n})$ such that the differences
\begin{equation}\label{eq-r2}
u \widetilde r_2-1, \widetilde r_2 u-1
\end{equation}
vanish in the domain $ \{|y|^2+|q|^2\ge R^2\}$.
Let us now consider the element
$$
 r=\widetilde r_1 \chi_1+ \widetilde r_2 \chi_2\in \mathcal{A}(\mathbb{R}^{4n}\cap\{|x|^2+|p|^2+|y|^2+|q|^2\ge 4R^2\} ),
$$
where $\chi_1,\chi_2\in C^\infty(\mathbb{R}^{4n}\cap\{|x|^2+|p|^2+|y|^2+|q|^2\ge 4R^2\})$  is a $U(n)$-invariant partition of unity associated with the covering  of the set $\{|x|^2+|p|^2+|y|^2+|q|^2\ge 4R^2\}$ by the domains 
$$
\mathbb{R}^{4n}\cap\{|x|^2+|p|^2+|y|^2+|q|^2\ge 4R^2\}\cap \{|x|^2+|p|^2\ge R^2\},
$$
$$
\mathbb{R}^{4n}\cap\{|x|^2+|p|^2+|y|^2+|q|^2\ge 4R^2\}\cap \{|y|^2+|q|^2\ge R^2\}.
$$
We claim that $r$ is the inverse of $u$ over the domain $\{|x|^2+|p|^2+|y|^2+|q|^2\ge 4R^2\}$.
Indeed, we have 
$$
ur=u\widetilde r_1 \chi_1+ u\widetilde r_2 \chi_2=(u\widetilde r_1-1) \chi_1+ (u\widetilde r_2-1) \chi_2+\chi_1+\chi_2
= 0+0+1=1.
$$
A similar computation shows that $ru=1$.

Thus, we proved that $(a\# b)^*(a\# b)$ and $(a\# b)(a\# b)^*$ are invertible in the corresponding $C^*$-algebras. Hence, by part~1 of the proof, the exterior product $a\# b$ is elliptic.
\end{proof}

\begin{remark}
One similarly defines the exterior product if  \emph{the first}
factor is equivariant. More generally, whenever we write an
expression of the form $a\#b$, we implicitly assume that one of the
factors is equivariant, and depending on which of the factors is
equivariant, we apply the corresponding definition. (If both
symbols are equivariant, we can use any of the definitions; both
give the same result.)
\end{remark}

\paragraph{2. Definition of the mapping $\ind$.} 
Given an elliptic $G$-operator  $(D,P_1,P_2)$  on $\mathbb{R}^n$, where
$$
D=\sum_g D_gR_g,\qquad P_j=\sum_g  P_{j,g} R_g ,\quad j=1,2,
$$
we now construct a $G$-operator 
acting in Hilbert modules over the group $C^*$-algebra $C^*(G)$ following the construction in~\cite[Sec.~5.2]{NaSaSt17}. To this end, let $L_g$ be the operator of left translation by $g$
in the free $C^*(G)$-module  $C^*(G)^N$. We define operators
\begin{equation}\label{tilde}   
     \widetilde D =\sum_g D_g R_g\otimes L_g, \qquad \widetilde P_{j}=\sum_g  P_{j,g} R_g\otimes L_g,\quad j=1,2,
\end{equation}
acting in the space $L^2(\mathbb{R}^n,C^*(G)^N)$. The operators are well defined by the universal property of the maximal crossed product. Then we consider the   operator 
\begin{equation}\label{volna-operator}
 \widetilde P_2\widetilde D \widetilde P_1: \im \widetilde P_1 \longrightarrow \im \widetilde P_2
\end{equation}
over the $C^*$-algebra $C^*(G)$ acting between the ranges of the projections
\begin{equation*}
    \widetilde P_j\colon L^2(\mathbb{R}^n,C^*(G)^N)\longrightarrow L^2(\mathbb{R}^n,C^*(G)^N)
\end{equation*}
considered as right Hilbert $C^*(G)$-modules. We claim that the operator \eqref{volna-operator} is $C^*(G)$-Fredholm in the sense of Mishchenko and Fomenko \cite{MiFo1}. Indeed, its almost-inverse operator is equal to $\widetilde P_1\widetilde{D^{-1}}\widetilde P_2$, where $D^{-1}$ is a $G$-operator with the symbol $r$, see~Definition~\ref{def-ell1}.
Thus, the operator \eqref{volna-operator} has an index
$$
 \ind_{C^*(G)} (\widetilde P_2\widetilde D \widetilde P_1: \im \widetilde P_1 \longrightarrow \im \widetilde P_2)  \in K_0(C^*(G)).
$$
Then we define
\begin{equation}\label{Bott-1}
  \begin{array}{ccc}
   \ind \colon K_0(C_0(T^*\mathbb{R}^n)\rtimes G) &  \longrightarrow & K_0(C^*(G)) \vspace{2mm} \\
  {}  [(\sigma(D),P_1,P_2)] &\longmapsto &  \ind_{C^*(G)} (\widetilde P_2\widetilde D \widetilde P_1: \im \widetilde P_1 \longrightarrow \im \widetilde P_2).
\end{array}
\end{equation} 

\paragraph{3. Definition of the mapping $\ind'$.} 

We define the index mapping
\begin{equation*}
    \ind'\colon K_0(C_0(\mathbb{R}^{4n})\rtimes G)\longrightarrow K_0(C_0(T^*\mathbb{R}^n)\rtimes G)
\end{equation*}
as follows. Let   $(x,p,y,q)$ be variables in $\mathbb{R}^{4n}$. Then each class in $K_0(C_0(\mathbb{R}^{4n})\rtimes G)$ contains a representative of the form
\begin{equation}\label{triple1}
(a,P_1,P_2),\qquad a\in \Mat_N(C (\mathbb{R}^{4n})\rtimes G), \quad P_{1,2}\in\Mat_N(C (T^*\mathbb{R}^n)\rtimes G),
\end{equation}
which is elliptic for large $(x,p,y,q)$ and such that
\begin{enumerate}
\item[(1)]   $ a(x,p,y,q)=P_1(x,p)=P_2(x,p)={\rm diag}(1,..,1,0,...,0)$ if $|x|^2+|p|^2\ge R^2$ for some $R>0$; 
\item[(2)]  $a(x,p,y,q)$  is homogeneous of degree zero in $(y,q)$ for $(y,q)$ large uniformly in $(x,p)$. 
\end{enumerate}
Such  a representative can be obtained if we use stable homotopies of the symbol and the projections. 
Note that here  we use the realization of the group $K_0(C_0(\mathbb{R}^{4n})\rtimes G)$ in terms of triples \eqref{triple1},
where the element $a$  defines the equivalence of projections at infinity (see Remark~\ref{diff1}).

We treat the triple in~\eqref{triple1} as a symbol of a $G$-operator and associate to it as in \eqref{tilde} the corresponding operator
\begin{equation}\label{goodop1}
\widetilde a\left(x,p,y,-i\frac\partial{\partial y}\right):\widetilde P_1 L^2(\mathbb{R}^n_y, (C_0(T^*\mathbb{R}^n)\rtimes G)^+\otimes\mathbb{C}^N)
\longrightarrow  \widetilde P_2 L^2(\mathbb{R}^n_y, (C_0(T^*\mathbb{R}^n)\rtimes G)^+\otimes\mathbb{C}^N) 
\end{equation}
acting in Hilbert $(C_0(T^*\mathbb{R}^n)\rtimes G)^+$-modules.
This operator is Fredholm with almost inverse operator defined by the triple $(a^{-1},P_2,P_1)$.
 Consider the index of this operator
\begin{equation}\label{eq-ind7}
 \ind_{C_0(T^*\mathbb{R}^n)\rtimes G}\widetilde a\left(x,p,y,-i\frac\partial{\partial y}\right) \in K_0(C_0(T^*\mathbb{R}^n)\rtimes G)
\end{equation} 
as the
$(C_0(T^*\mathbb{R}^n)\rtimes G)^+$-index of  operator \eqref{goodop1}.  \emph{A priori} this index lies in
$K_0((C_0(T^*\mathbb{R}^n)\rtimes G)^+)$, but one can readily show that the homomorphism
$(C_0(T^*\mathbb{R}^n)\rtimes G)^+\to\mathbb{C}$, whose kernel is $C_0(T^*\mathbb{R}^n)\rtimes G$, takes the operator \eqref{goodop1} to the
identity operator (by our assumption (1) above), whose index is zero, and hence
\begin{equation*}
 \ind_{(C_0(T^*\mathbb{R}^n)\rtimes G)^+}\widetilde a\left(\!x,p,y,-i\frac\partial{\partial y}\right)\in K_0(C_0(T^*\mathbb{R}^n)\rtimes G)\equiv \ker
 \bigl(K_0((C_0(T^*\mathbb{R}^n)\rtimes G)^+)\to K_0(\mathbb{C}) \bigr).
\end{equation*}
Finally, we define $\ind'[(a,P_1,P_2)]$ as the index \eqref{eq-ind7}.

\paragraph{4. Proof of \eqref{q1}.}

Let us prove that $\ind\circ\beta=I$. To this end, note
that if $[P]\in K_0( C^*(G))$, where $P$ is a projection over $C^*(G)$, then the class $\beta[P]\in K_0(C_0(T^*\mathbb{R}^n)\rtimes G)$ is
represented by the elliptic symbol
$$
(\sigma(\mathcal{E}_0\otimes 1_N),1\otimes P,1\otimes P).
$$
Hence $\ind \beta[P]$ is equal to the $C^*(G)$-index
of the operator
\begin{equation}\label{uuuu}
     \widetilde{\mathcal{E}}_0\otimes 1_N\colon
   1\otimes \widetilde P L^2(\mathbb{R}^n,\Lambda^{ev}(\mathbb{C}^n) \otimes C^*(G)^N)
    \longrightarrow
    1\otimes \widetilde P L^2(\mathbb{R}^n, \Lambda^{odd}(\mathbb{C}^n)\otimes C^*(G)^N).
\end{equation}
However, the cokernel of $\mathcal{E}_0$ is trivial, and the kernel is
one-dimensional and consists of $G$-invariant elements. Thus, the
cokernel of the operator~\eqref{uuuu} is trivial  and the kernel is
\begin{equation*}
    \ker 1\otimes \widetilde P(\widetilde{\mathcal{E}}_0\otimes 1_N)=\ker \mathcal{E}\otimes \im \widetilde P \simeq\im P\subset C^*(G)^N.
\end{equation*}
We obtain the desired equality
\begin{equation} \label{bott11}
    \ind \beta[P]=[P].
\end{equation}

\paragraph{5. Proof of \eqref{q2}.}

The proof is similar to that in \cite{NaSaSt17}. For the sake of completeness, let us give a shorter proof here.
Given an arbitrary element in $K_0(C_0(T^*\mathbb{R}^n)\rtimes G)$, we choose its  representative of the form
$$
 {\bf a}=(a,P_1,P_2),\quad a,P_1,P_2\in C (T^*\mathbb{R}^n,\Mat_N(\mathbb{C}))\rtimes G,
$$
where $P_1=P_2=a={\rm diag}(1,1,..,1,0,0,..,0)$ in the domain $|x|^2+|p|^2\ge R^2$ for some $R>0$. Then in the class of the element $\beta'[{\bf a}]$ we choose the following representative
\begin{equation}\label{bott17}
a\# \sigma(\mathcal{E}_0)=
 \left(
    \begin{array}{cc}
       a(x,p)\otimes 1 & -\chi(x,p )  (1\otimes \sigma^*(\mathcal{E}_0)(y,q)) \\
       \chi(x,p )  (1\otimes \sigma(\mathcal{E}_0)(y,q))  & a^*(x,p)\otimes 1
   \end{array}
 \right)
\end{equation}
where $\chi(x,p)$ is a smooth $U(n)$-invariant function with compact support on $T^*\mathbb{R}^n$ such that $\chi(x,p)\equiv 1$ 
whenever $|x|^2+|p|^2\le R^2$. Furthermore, we  suppose that here $\sigma(\mathcal{E}_0)(y,q)$ is homogeneous at infinity and
continuous at $y=q=0$.  Clearly, this representative satisfies the properties in the definition of the mapping $\ind'$.
Hence, we have by the definition of the mapping $\ind'$ the following equality
$$
 \ind'\beta'[{\bf a}]=\ind_{(C_0(T^*\mathbb{R}^n)\rtimes G)^+}\widetilde{A},
$$
where $\widetilde{A}$  is an operator in Hilbert $(C_0(T^*\mathbb{R}^n)\rtimes G)^+$-modules  associated with the symbol
\eqref{bott17}. We make the following choice of $\widetilde{A}$ :
\begin{multline}
\widetilde{A}=
 \left(
    \begin{array}{cc}
       \widetilde{a}(x,\xi)\otimes (1-\Pi) & -\chi(x,\xi )  (1\otimes  \mathcal{E}_0^*)) \\
       \chi(x,\xi )  (1\otimes \mathcal{E}_0)  & \widetilde{a}^*(x,\xi)\otimes 1
   \end{array}
 \right):\\
\begin{array}{c}
\widetilde{P}_1 L^2(\mathbb{R}^n_y, (C_0(T^*\mathbb{R}^n)\rtimes G)^+\otimes \mathbb{C}^N\otimes \Lambda^{ev}(\mathbb{C}^n)) \\
\oplus \\
\widetilde{P}_2 L^2(\mathbb{R}^n_y, (C_0(T^*\mathbb{R}^n)\rtimes G)^+\otimes \mathbb{C}^N\otimes\Lambda^{odd}(\mathbb{C}^n))
\end{array}
\longrightarrow\\
\begin{array}{c}
\widetilde{P}_2 L^2(\mathbb{R}^n_y, (C_0(T^*\mathbb{R}^n)\rtimes G)^+\otimes \mathbb{C}^N\otimes\Lambda^{ev}(\mathbb{C}^n)) \\
\oplus \\
\widetilde{P}_1 L^2(\mathbb{R}^n_y, (C_0(T^*\mathbb{R}^n)\rtimes G)^+\otimes \mathbb{C}^N\otimes\Lambda^{odd}(\mathbb{C}^n))
\end{array}
\end{multline}
where $\Pi$ is the orthogonal projection on the subspace  $\ker\mathcal{E}_0=\mathbb{C}e^{-|y|^2/2}$.

Then we have
\begin{multline}
\ker \widetilde{A}=\ker \widetilde{A}^*\widetilde{A}=\\
\ker {\rm diag}\Bigl(
     \widetilde{a}^* \widetilde{a} (x,p)\otimes (1-\Pi)+  \chi^2(x,p )  (1\otimes \mathcal{E}^*_0 \mathcal{E}_0),
 \widetilde{a} \widetilde{a}^*(x,p)\otimes 1 +  \chi^2(x,p )  (1\otimes \mathcal{E}_0 \mathcal{E}^*_0)
\Bigr )
\end{multline}
We claim that the operator 
$$
 (\widetilde{a} \widetilde{a}^*)(x,p)\otimes 1 +  \chi^2(x,p )  (1\otimes \mathcal{E}_0 \mathcal{E}^*_0))
$$
is strictly positive and, hence, invertible. Indeed, this operator is a sum of two nonnegative operators and for $|x|^2+|p|^2\le R^2$ the second summand is strictly positive since $\ker\mathcal{E^*_0}=0$, while 
for $|x|^2+|p|^2\ge R^2$ the first term is strictly positive, since $\widetilde{a} $ is invertible here. One shows similarly that the kernel of operator
$$
  (\widetilde{a}^* \widetilde{a})(x,p)\otimes (1-\Pi)+  \chi^2(x,p )  (1\otimes \mathcal{E}^*_0 \mathcal{E}_0)
$$
is equal to $\im P_1\otimes \ker \mathcal{E}_0\simeq \im P_1$ and this operator is strictly positive on the orthogonal complement of this subspace.  Thus, we have
$$
\ker \widetilde{A}= (\im P_1\otimes \ker \mathcal{E}_0)\oplus 0\simeq \im P_1.
$$
The kernel of the adjoint operator is similarly equal to
$$
\ker \widetilde{A}^*=\ker  \widetilde{A}\widetilde{A}^*=(\im P_2\otimes \ker \mathcal{E}_0)\oplus 0\simeq \im P_2.
$$
Hence, we obtain
$$
 \ind_{(C_0(T^*\mathbb{R}^n)\rtimes G)^+}\widetilde{A}=[\ker\widetilde{A}]-[\ker\widetilde{A}^*]=
[P_1]-[P_2]\in K_0(C_0(T^*\mathbb{R}^n)\rtimes G).
$$
This proves \eqref{q2}.

\paragraph{6. Proof of \eqref{q3}.}
Given $[a]\in K_0(C_0(T^*\mathbb{R}^n)\rtimes G)$, we claim that the element $a\#\sigma(\mathcal{E}_0)$ is homotopic within
elliptic symbols to an element unitarily equivalent to $\sigma(\mathcal{E}_0)\#a$. Indeed, the
homotopy
\begin{equation*}
    \sigma_t=a(x\cos t + y\sin t,p\cos t +q\sin t)\#
          \sigma(\mathcal{E}_0)(y\cos t - x\sin t,q\cos t - p\sin t)
\end{equation*}
for $t\in[0,\pi/2]$ takes $a(x,p)\#\sigma(\mathcal{E}_0)(y,q)$ to
$a(y,q)\#\sigma(\mathcal{E}_0)(-x,-p)$, and then the $180^\circ$ rotation in the
$(x,p)$-plane takes it to the symbol unitarily equivalent to
$\sigma(\mathcal{E}_0)\#a$. Moreover, this homotopy preserves the ellipticity of the symbol, since the diagonal action of $G$ on $\mathbb{R}^{4n}$ commutes  with the rotation homotopy
$$
(x,p,y,q) \longmapsto (x\cos t + y\sin t,p\cos t +q\sin t,y\cos t - x\sin t,q\cos t - p\sin t).
$$

Finally, the following equality holds 
\begin{equation}\label{bott3}
    \ind_{(C_0(T^*\mathbb{R}^n)\rtimes G)^+}[\sigma(\mathcal{E}_0)\#a]=\beta\ind [a].
\end{equation}
The proof of this equality  coincides with the proof of Lemma~6.7 in \cite{NaSaSt17}.

\end{proof}

\section{Proof of the Index Formula}

Both sides of the index formula \eqref{indf1} are homomorphisms of Abelian groups 
$$
 \ind,\ind_t: \Ell(\mathbb{R}^n,G)\longrightarrow \mathbb{C}.
$$
The group  $\Ell(\mathbb{R}^n,G)\simeq K_0(C_0(T^*\mathbb{R}^n)\rtimes G)$ is generated by the stable homotopy classes of twisted Euler operators \eqref{euler3} by the equivariant Bott periodicity (see Theorem~\ref{bott2}). 
Hence, it suffices to prove that the analytic index is equal to the topological index for the twisted Euler operators.

\paragraph{The analytic index of twisted Euler operators.} The cokernel is trivial (this follows from the fact that 
$\mathcal{E}_0\otimes 1_N$ is surjective and commutes with $1\otimes P$), while the kernel
is equal to $P\mathbb{C}^N \exp(-r^2/2)$. Hence,
\begin{equation}\label{eq-z1}
 \ind(\mathcal{E}_0\otimes 1_N,1\otimes P,1\otimes P)=\rk P|_{\mathbb{C}^N \exp(-r^2/2)}=\tr P|_{\mathbb{C}^N \exp(-r^2/2)}= \sum_{g\in G} {\rm tr} P_g=: \sum_{\langle g\rangle \subset G}\ch_g[P].
\end{equation}
Here $\tr$ stands for the operator trace on $L^2(\mathbb{R}^n,\mathbb{C}^N)$,   ${\rm tr}$ is the matrix trace,   $P_g$ are the components of $P\in \Mat_N(C^\infty(G))$, and we used the fact that the Gaussian function $\exp(-r^2/2)$ is $U(n)$-invariant.

\paragraph{The topological index of   twisted Euler operators.}
 Given $g\in G$, let us compute the localized topological index
$\ind_t [\sigma (\mathcal{E}_0\otimes 1_N,1\otimes P,1\otimes P)](g)$.  Let $P_1$ and $P_0$ be matrix projections over $C^\infty(\mathbb{C}^n)\rtimes G$ such that
$$
[\sigma (\mathcal{E}_0\otimes 1_N,1\otimes P,1\otimes P)]=[P_1]-[P_0].
$$
By the definition of the localized topological index, we have
\begin{equation}\label{new5}
 \ind_t [\sigma (\mathcal{E}_0\otimes 1_N,1\otimes P,1\otimes P)](g)=\frac{1}{\det(1-g|_{L_g^\perp})}\sum_{s\in\langle g\rangle}{\rm tr}(\tau_s(\omega_s))=
 \frac{1}{\det(1-g|_{L_g^\perp})}    
\sum_{s\in\langle g\rangle} \int_{L_s } {\rm tr}(\omega_s|_{L_s }), 
\end{equation}
where the functional $\tau_s$ was defined in \eqref{geom1}, $L=\mathbb{C}^n$,  $L_s$ is the fixed-point subspace for $s\in U(n)$, and we set
$$
\omega=\{\omega_s\}_{s\in G}=P_1\exp\left(-\frac{ dP_1dP_1}{2\pi i}\right)-P_0\exp\left(-\frac{ dP_0dP_0}{2\pi i}\right)
\in \Mat_N(C^\infty_c(\mathbb{C}^n,\Lambda(\mathbb{C}^n)) \rtimes G).
$$
We claim that the following equality holds
\begin{equation}\label{newold1}
\sum_{s\in\langle g\rangle} \int_{L_s }{\rm tr} (\omega_s|_{L_s })=
\int_{L_g}  \ch_{g}[\sigma (\mathcal{E}_0\otimes 1_N,1\otimes P,1\otimes P)]
\end{equation}
where $ \ch_{g}[\sigma (\mathcal{E}_0\otimes 1_N,1\otimes P,1\otimes P)]\in H^{ev}_c(L_g)$  is the localized Chern character of the symbol of the twisted Euler operator defined in \cite[p.92]{NaSaSt17}. Indeed, it follows from the definitions in the cited monograph that 
\begin{equation}\label{newold2}
 \ch_{g}[\sigma (\mathcal{E}_0\otimes 1_N,1\otimes P,1\otimes P)]=
\sum_{s\in\langle g\rangle}\;\;\;\int_{ \overline G_{g,s}} {\rm tr} (h^* \omega_s)|_{L_g }dh,
\end{equation}
where $\overline G_{g,s}=kC_g\subset U(n)$, $C_g=\{h\in U(n)\;|\; gh=hg\}$ is the centralizer of $g$ in $U(n)$ (it is a compact Lie group), and $k$ is an arbitrary element such that $kgk^{-1}=s$. Finally, $dh$ is the measure on $\overline G_{g,s}$ induced by the element $k$ from the normalized Haar measure on $C_g$. Integrating \eqref{newold2} over $L_g$ gives us the desired equality:
\begin{multline*}
\int_{L_g}\ch_{g}[\sigma (\mathcal{E}_0\otimes 1_N,1\otimes P,1\otimes P)]=
\sum_{s\in\langle g\rangle}\;\;\;\int_{ \overline G_{g,s}}\int_{L_g} {\rm tr} (h^* \omega_s)|_{L_g }dh\\
=\sum_{s\in\langle g\rangle}\;\;\;\int_{ \overline G_{g,s}}\int_{L_s} {\rm tr} ( \omega_s|_{L_s })dh=
\sum_{s\in\langle g\rangle}\;\;\; \int_{L_s} {\rm tr} ( \omega_s|_{L_s }). 
\end{multline*}
Here we used the fact that each $h\in  \overline G_{g,s}$  defines a diffeomorphism $h:L_g\to L_s$ of the fixed-point sets of $g$ and $s$.
Thus, Eqs. \eqref{new5} and \eqref{newold1} give us the following equality
\begin{equation}\label{alpha1}
\ind_t [ (\mathcal{E}_0\otimes 1_N,1\otimes P,1\otimes P)](g)= \frac{1}{\det(1-g_{L_g^\perp})}  \int_{L_g}  \ch_{g}[\sigma (\mathcal{E}_0\otimes 1_N,1\otimes P,1\otimes P)].
\end{equation}
The localized Chern character is multiplicative (see  \cite[Lemma~9.10]{NaSaSt17})  and we have
\begin{equation}\label{alpha11}
   \ch_{g}[\sigma (\mathcal{E}_0\otimes 1_N,1\otimes P,1\otimes P)]=
  \ch_g[P] \ch i^*(\sigma(\mathcal{E}_0))(g),
\end{equation}
where $\ch_g[P]=\sum_{s\in\langle g\rangle}{\rm tr}P_s\in \mathbb{C}$ and   $i^*(\sigma(\mathcal{E}_0))$ is the restriction of symbol $\sigma(\mathcal{E})_0$ to the subspace $L_g$.

A direct computation shows that the restriction of  the symbol of the Euler operator to the fixed-point set is equal to  
$$
i^*\sigma(\mathcal{E}_0)=(1_{\Lambda^{ev}(L^\perp_{g} )} \otimes \sigma(\mathcal{E}_{L_g}))\oplus (1_{\Lambda^{odd}(L^\perp_{g} )} \otimes \sigma(\mathcal{E}^*_{L_g})),
$$
where $\Lambda^{ev/odd}(L^\perp_{g} )$ are the vector spaces of even/odd algebraic forms of $L^\perp_{g}$, and we denote  
the symbol of the Euler operator on a vector space $L$ by $\sigma(\mathcal{E}_{L})$. 
Now note that the action of $g$ is nontrivial only on the exterior algebra of   $L^\perp_g$. Hence, the localized  Chern character is equal to\footnote{Recall the definition of the localized Chern character for a trivial $G$-space $X$:
$$
\ch(\cdot)(g): K_G(X)\simeq K(X)\otimes R(G) \stackrel{\ch\otimes{\rm tr}_g}\longrightarrow H^*(X)\otimes\mathbb{C},
$$
where $K_G(X)\simeq K(X)\otimes R(G) $ is the natural isomorphism, $R(G)$ is the ring of virtual representations of $G$, $\ch$ is the Chern character, while ${\rm tr}_g:R(G)\to \mathbb{C}$ takes a virtual representation to the value of its character at the element $g\in G$.
}
$$
\ch (i^* \sigma(\mathcal{E}_0) )(g)={\rm tr}_g([\Lambda^{ev}(L^\perp_{g} )(g)]-[\Lambda^{odd}(L^\perp_{g} )] )
\cdot \ch (\sigma(\mathcal{E}_{L_g}))= \det(1-g_{L_g^\perp}) \cdot \ch (\sigma(\mathcal{E}_{L_g})). 
$$
This equality follows from the definition of the localized Chern character and the fact that ${\rm tr}_g([\Lambda^{ev}(L^\perp_{g} )(g)]-[\Lambda^{odd}(L^\perp_{g} )] )= \det(1-g_{L_g^\perp}) $, which is easy to see if we diagonalize $g_{L_g^\perp}$. Substituting the expression for the localized Chern character in \eqref{alpha1}, we obtain
\begin{equation}\label{alpha1s}
\ind_t [\sigma(\mathcal{E}_0\otimes 1_N,1\otimes P,1\otimes P)](g)= 
\ch_g[P] \frac{\det(1-g_{L_g^\perp})}{\det(1-g_{L_g^\perp})}  \int_{L_g}  \ch  (\sigma(\mathcal{E}_{L_g})).
\end{equation}
Here, the determinants cancel, while the integral is well known and is equal to one 
\begin{align*}
 \int_{L_g}  \ch  (\sigma(\mathcal{E}_{L_g}))=\left(\int_{\mathbb{C}} \ch (\sigma(\mathcal{E}_{\mathbb{C}}))\right)^{\dim L_g} =1.
\end{align*}
This equality is a special case of Riemann--Roch formula for the embedding $pt\subset L_g$, see  e.g.~\cite{LuMi1}.
Hence  we obtain the formula for the localized topological index of the twisted Euler operator
$$
\ind_t [\sigma(\mathcal{E}_0\otimes 1_N,1\otimes P,1\otimes P)](g)=\ch_g(P).
$$
Then the topological index itself is equal to
\begin{equation}\label{ind-tt}
 \ind_t [\sigma(\mathcal{E}_0\otimes 1_N,1\otimes P,1\otimes P)] =\sum_{\langle g\rangle \subset G}\ch_g[P].
\end{equation}
Comparing the expressions for the analytic index in \eqref{eq-z1} and the topological index in \eqref{ind-tt} we see that they are equal. The proof of the index formula is now complete.

\end{document}